\documentclass[10pt,reqno]{amsart}

\usepackage{amsmath}
\usepackage{amsthm}
\usepackage{amsfonts}
\usepackage{amssymb}
\usepackage{graphicx}
\usepackage{tikz}
\usepackage{hyperref}
\usepackage[margin=1.0in]{geometry}

\usepackage{todonotes}

 \definecolor{skyblue}{rgb}{0.85,0.85,1}

\newtheorem{lemma}{Lemma}
\newtheorem{prop}{Proposition}
\newtheorem{theorem}{Theorem}
\newtheorem{cor}{Corollary}
\newtheorem{rem}{Remark}

\DeclareMathOperator{\tr}{Tr}
\DeclareMathOperator{\dist}{dist}

\DeclareMathOperator{\Mor}{Mor}
\DeclareMathOperator{\Mas}{Mas}

\newcommand{\bbR}{\mathbb{R}}
\newcommand{\bbS}{\mathbb{S}}
\newcommand{\bbT}{\mathbb{T}}

\newcommand{\cJ}{\mathcal{J}}
\newcommand{\cH}{\mathcal{H}}
\newcommand{\cD}{\mathcal{D}}
\newcommand{\cF}{\mathcal{F}}

\newcommand{\Hh}{H^{1/2}(\Sigma) \oplus H^{-1/2}(\Sigma)}
\newcommand{\Hp}{H^{1/2}(\Sigma)}
\newcommand{\Hm}{H^{-1/2}(\Sigma)}

\newcommand{\pO}{\partial \Omega}
\newcommand{\pM}{\partial M}
\newcommand{\ra}{\rightarrow}
\newcommand{\lra}{\longrightarrow}
\newcommand{\p}{\partial}
\newcommand{\tL}{\widetilde{\Lambda}}

\newcommand*{\longhookrightarrow}{\ensuremath{\lhook\joinrel\relbar\joinrel\rightarrow}}

\begin{document}

\title{Manifold decompositions and indices of Schr\"{o}dinger operators}

\author{Graham Cox}\email{ghc5046@psu.edu}\address{Penn State University Mathematics Dept., University Park, State College, PA 16802}
\author{Christoper K.R.T. Jones}\email{ckrtj@email.unc.edu}
\author{Jeremy L. Marzuola}\email{marzuola@email.unc.edu}\address{Department of Mathematics, UNC Chapel Hill, Phillips Hall CB \#3250, Chapel Hill, NC 27599}


\begin{abstract}
The Maslov index is used to compute the spectra of different boundary value problems for Schr\"{o}dinger operators on compact manifolds. The main result is a spectral decomposition formula for a manifold $M$ divided into components $\Omega_1$ and $\Omega_2$ by a separating hypersurface $\Sigma$. A homotopy argument relates the spectrum of a second-order elliptic operator on $M$ to its Dirichlet and Neumann spectra on $\Omega_1$ and $\Omega_2$, with the difference given by the Maslov index of a path of Lagrangian subspaces. This Maslov index can be expressed in terms of the Morse indices of the Dirichlet-to-Neumann maps on $\Sigma$. Applications are given to doubling constructions, periodic boundary conditions and the counting of nodal domains. In particular, a new proof of Courant's nodal domain theorem is given, with an explicit formula for the nodal deficiency.
\end{abstract}

\maketitle

\section{Introduction}
Suppose $M$ is a compact, orientable manifold, and $L$ a selfadjoint, elliptic operator on $M$. It is of great interest to compute the spectrum of $L$, given boundary conditions on $\p M$, and relate it to the underlying geometry of $M$ and $L$. Of particular importance for many applications is the \emph{Morse index}, or number of negative eigenvalues. One approach to computing the Morse index is to deform $M$ through a one-parameter family of domains $\{\Omega_t\}$ and keep track of eigenvalues passing through $0$ as $t$ varies. For instance, if $f\colon M \to [0,1]$ is a Morse function on $M$ with $f^{-1}(1) = \p M$, one can consider the sublevel sets $\Omega_t = f^{-1}[0,t)$ for $t \in (0,1]$. Since
\[
	\lim_{t\ra0} \operatorname{Vol}(\Omega_t) = 0,
\]
it is easy to compute the Morse index of $L$ on $\Omega_t$ once $t$ is sufficiently small. It thus remains to describe the ``spectral flow" of the boundary value problems as $t$ varies.

This was done by Smale in \cite{S65} for the Dirichlet problem, assuming the $\Omega_t$ remain diffeomorphic for all $t$, which is the case when $f$ has no critical values in $(0,1]$. An application of this result to the study of minimal surfaces was given by Simons in \cite{S68}, and a generalization was given by Uhlenbeck in \cite{U73}, allowing the topology of $\Omega_t$ to change but still assuming Dirichlet boundary conditions. In \cite{DJ11} Deng and Jones used the Maslov index, a symplectic invariant, to generalize Smale's result to more general boundary conditions, but with the additional requirement that the domain be star-shaped. The star-shaped restriction was subsequently removed in \cite{CJM14}, where the Maslov index was used to compute the spectral flow for any smooth one-parameter family of domains, with quite general boundary conditions.

To complete the picture, we must describe what happens when $t$ passes through a critical value of $f$ and the topology of $\Omega_t$ changes. More generally, we consider a decomposition of $M$ into disjoint components along a separating hypersurface $\Sigma$, as in Figure \ref{decomposition}, and ask how the spectrum on $M$ relates to the spectrum on each component. This question is answered in Theorem \ref{MorMas}, which says the Morse index of $L$ on $M$ equals the sum of the Morse indices on each component (with appropriate boundary conditions on $\Sigma$) plus a ``topological contribution," which is given by the Maslov index of a path of Lagrangian subspaces in the symplectic Hilbert space $\Hp \oplus \Hm \oplus \Hp \oplus \Hm$.

By considering the limit in which either $\Omega_1$ or $\Omega_2$ is small in some appropriate sense, this can be related to classic results on eigenvalues of the Laplacian with respect to singular perturbations of the underlying spatial domain. In \cite{CF78} Chavel and Feldman studied the effect of removing a tubular neighborhood of a closed submanifold $N \subset M$, replacing $M$ by the domain
\[
	M_\epsilon = \{x \in M : \dist(x,N) > \epsilon\}.
\]
Assuming the codimension of $N$ is at least 2, they proved convergence of the Dirichlet spectrum on $M_\epsilon$ to the spectrum on $M$ as $\epsilon \ra 0$. A similar analysis was carried out in \cite{CF81} for manifolds to which a small handle has been attached, with a sufficient condition given, in terms of an isoperimetric constant, for convergence of the spectrum as the size of the handle decreases to zero. In \cite{RT75} Rauch and Taylor considered a rather weak notion of convergence for Euclidean domains and described the behavior of the Laplacian when one removes a small neighborhood of a polar set. (The definition of a polar set can be found in \cite{RT75}; note in particular that a submanifold of codimension at least 2 is a polar set, whereas a hypersurface is not.) In \cite{J89} Jimbo considered the case of two disjoint bounded domains connected by a small tube, and gave asymptotic formulas for the eigenvalues and eigenfunctions in the singular limit as the tube shrinks to a line.

Several other authors have considered the reduction of spectral flows (and other analytic invariants) through similar manifold decompositions \cite{CLM96,N95,Y91}. These results are all for first-order, Dirac-type elliptic operators which have a particular form in a collar neighborhood of the separating hypersurface $\Sigma$.

Our symplectic approach to this problem has many applications, which we explore in the last section of the paper. The first is a new proof of Courant's nodal domain theorem, with an explicit formula for the nodal deficiency. Then we compute the Morse indices of operators on ``almost-doubled" manifolds, which consist of two identical (or almost identical) components glued together along a common boundary. We also use the Maslov index to give a new proof of a well-known theorem relating the Dirichlet and Neumann counting functions to the spectrum of the Dirichlet-to-Neumann map. Finally, we relate the spectra of Schr\"odinger operators on the torus---viewed as a cube with opposing faces identified---to the spectra on the cube with Dirichlet boundary conditions, and find that the periodic and Dirichlet Morse indices are related by a kind of symmetrized Dirichlet-to-Neumann map.

\subsection*{Structure of the paper}
In Section \ref{secResults} we define the relevant operators and domains, and state the main results of the paper. The fundamental relation between Morse and Maslov indices is proved in Section \ref{secProof}. In Section \ref{secCrossing} we study the Maslov index in more detail, and relate it to the Dirichlet-to-Neumann maps of the manifold decomposition. Finally, in Section \ref{secApp} these results are applied to a variety of geometric scenarios.

\section*{Acknowledgments} The authors wish to thank Chris Judge, Rafe Mazzeo, Michael Taylor and Gunther Uhlmann for very helpful conversations during the preparation of this manuscript.
JLM was supported in part by U.S. NSF DMS-1312874 and NSF CAREER Grant DMS-1352353.  GC and CKRTJ were supported by U.S. NSF Grant DMS-1312906.


\section{Definitions and results}\label{secResults}
Throughout we assume that $M$ is a compact, orientable manifold with Lipschitz boundary $\pM$, and $\Sigma \subset M$ is an embedded Lipschitz hypersurface that separates $M$ into two disjoint (but not necessarily connected) components: $M \setminus \Sigma = \Omega_1 \cup \Omega_2$. We further assume that $\Sigma \cap \pM = \varnothing$. A typical situation is shown in Figure \ref{decomposition}.

\begin{center}
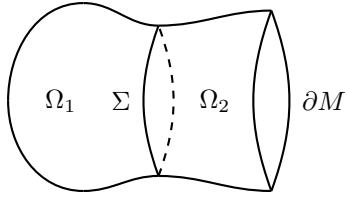
\begin{figure}
\begin{tikzpicture}
	\draw[thick] (0,1) to [out=270, in=180] (1,-0.2); 
	\draw[thick] (1,-0.2) to [out=0, in=180] (2,0); 
	\draw[thick] (0,1) to [out=90, in=180] (1,2.3); 
	\draw[thick] (1,2.3) to [out=0, in=180] (2,2); 
	\draw[thick] (2,2) to [out=0, in=180] (3.5,2.2); 
	\draw[thick] (2,0) to [out=0, in=180] (3.5,-0.2); 
	\draw[thick] (2,0) to [out=110, in=250] (2,2); 
	\draw[thick, dashed] (2,0) to [out=70, in=290] (2,2); 
	\draw[thick] (3.5,-0.2) to [out=110, in=250] (3.5,2.2); 
	\draw[thick] (3.5,-0.2) to [out=70, in=290] (3.5,2.2); 
	\node at (0.7,1) {$\Omega_1$}; 
	\node at (1.5,1) {$\Sigma$};
	\node at (2.75,1) {$\Omega_2$};
	\node at (4.2,1) {$\p M$};
\end{tikzpicture}
\caption{A manifold $M$ with nontrivial boundary $\pM$, separated into components $\Omega_1$ and $\Omega_2$ by an orientable hypersurface $\Sigma$.}
\label{decomposition}
\end{figure}
\end{center}

Let $g$ be a Riemannian metric on $M$ and $V$ a real-valued function, both of class $L^\infty$. We define the formal differential operator
\begin{align}\label{Ldef}
	L = -\Delta_g + V,
\end{align}
where $\Delta_g$ is the Laplace--Beltrami operator of $g$. (This is a formal operator in the sense that its domain has not been specified; we will allow $L$ to act on functions on $\Omega_1$, $\Omega_2$ and $M$.)  

Fix $i \in\{1,2\}$ and suppose $u \in H^1(\Omega_i)$ and $\Delta_g u \in L^2(\Omega_i)$. It follows from Theorem 3.37 and Lemma 4.3 in \cite{M00} that 
\begin{align}\label{trace}
	\left.u\right|_{\pO_i} \in H^{1/2}(\pO_i), \quad \left.\frac{\p u}{\p \nu_i}\right|_{\pO_i} \in H^{-1/2}(\pO_i)
\end{align}
and so the following weak version of Green's first identity
\begin{align}\label{Green}
	\int_{\Omega_i}  \left<\nabla u, \nabla v\right> = -\int_{\Omega_i} (\Delta_g u)v + \int_{\pO_i} v \frac{\p u}{\p \nu_i}
\end{align}
holds for any $v \in H^1(\Omega_i)$, where $\nu_i$ denotes the outward unit normal to $\Omega_i$. Note that $\pO_i$ is the disjoint union $\Sigma \cup (\pO_i \cap \pM)$ and $\nu_1 = -\nu_2$ on $\Sigma$.

We suppose that either Dirichlet or Neumann boundary conditions are prescribed on each connected component of $\p M$, and correspondingly write $\p M = \Sigma_D \cup \Sigma_N$. (Thus $\Sigma_D$ and $\Sigma_N$ are closed, disjoint Lipschitz hypersurfaces which need not be connected.) 

For $i \in \{1,2\}$ we let $L^D_i$ and $L^N_i$ denote the Dirichlet and Neumann realizations of $L$ on $\Omega_i$, respectively. These are unbounded, selfadjoint operators on $L^2(\Omega_i)$, with domains
\begin{align}
	\cD(L^D_i) &= \left\{u \in H^1(\Omega_i) : \Delta_g u \in L^2(\Omega_i), \left. u \right|_{\Sigma \cup (\Sigma_D \cap \pO_i)} =0 \text{ and } \left.\frac{\p u}{\p \nu_i}\right|_{\Sigma_N \cap \pO_i} = 0 \right\} \label{LDdef} \\
	\cD(L^N_i) &= \left\{u \in H^1(\Omega_i) : \Delta_g u \in L^2(\Omega_i), \left. u \right|_{\Sigma_D \cap \pO_i} = 0 \text{ and } \left.\frac{\p u}{\p \nu_i}\right|_{\Sigma \cup (\Sigma_N \cap \pO_i)} = 0 \right\}.\label{LNdef} 
\end{align}
The operators $L^D_i$ and $L^N_i$ have the same boundary conditions on the ``outer boundary" $\p M \cap \pO_i$; the superscript refers only to the conditions imposed on the ``inner boundary" $\Sigma$. We let $L^G$ denote the ``global" realization of $L$ on $M$. This is an unbounded, selfadjoint operator on $L^2(M)$, with domain
\begin{align}
	\cD(L^G) = \left\{u \in H^1(M) : \Delta_g u \in L^2(M), \left.u\right|_{\Sigma_D \cap \p\Omega_i} = 0 
	\text{ and } \left.\frac{\p u}{\p \nu}\right|_{\Sigma_N \cap \p\Omega_i} = 0 \right\}.\label{LGdef} 
\end{align}

Each of the operators $L^G$, $L^D_i$ and $L^N_i$ is bounded below and selfadjoint with compact resolvent, and therefore has a well-defined \emph{Morse index} (number of negative eigenvalues, counting multiplicity), which we denote $\Mor(\cdot)$. We additionally let 
\begin{equation}
\label{Mor0}
\Mor_0 = \Mor + \dim\ker
\end{equation} 
denote the number of nonpositive eigenvalues.

Our main result relates the Morse index of $L^G$ to the Morse indices of $L^D_i$ and $L^N_i$. We compare these quantities by encoding the boundary conditions on $\Sigma$ in a Lagrangian subspace, which is then rotated between global boundary conditions, corresponding to $L^G$, and decoupled boundary conditions, corresponding to $L^N_1$ and $L^D_2$; see \eqref{betadef} below. The difference in Morse indices is equated to a symplectic winding number---the Maslov index---for the rotating path of boundary conditions. The relevant technical properties of the Maslov index are summarized in Appendix B of \cite{CJM14}; a more complete presentation can be found in \cite{BF98} or \cite{F04}. Some applications of the Maslov index to boundary value problems for PDE can be found in \cite{CJM14,CJLS14,DJ11,LSS14,PWindex,S78II}.

Consider the Hilbert space $\cH = \Hh$, with the symplectic form $\omega$ induced by the bilinear pairing of $H^{1/2}(\Sigma)$ with $(H^{1/2}(\Sigma))^* = H^{-1/2}(\Sigma)$, that is
\[
	\omega((x,\phi),(y,\psi)) = \psi(x) - \phi(y)
\]
for $x,y \in \Hp$ and $\phi,\psi \in \Hm$. To study the decomposition of $M$ by $\Sigma$ we use the doubled space 
\begin{equation}
\label{Hbpdef}
\cH_\boxplus := \cH \oplus \cH,
\end{equation} 
with the symplectic form $\omega_\boxplus := \omega \oplus(-\omega)$. The negative sign on the second component is chosen so the diagonal subspace $\{(x,\phi,x,\phi) : x \in \Hp, \phi \in \Hm\}$ is Lagrangian.

For each $i \in \{1,2\}$ and $\lambda \in \bbR$ we define
\begin{align}\label{Kdef}
	K_i^\lambda = \left\{ u \in H^1(\Omega_i) : Lu = \lambda u, \left.u\right|_{\Sigma_D \cap \p \Omega_i} = 0 \text{ and } \left.\frac{\p u}{\p \nu_i} \right|_{\Sigma_N \cap \Omega_i} = 0 \right\},
\end{align}
with the equality $Lu = \lambda u$ meant in the distributional sense. Thus $K_i^\lambda$ is the space of weak $H^1(\Omega_i)$ solutions to $Lu = \lambda u$ that satisfy the given boundary conditions on $\p M \cap \p \Omega_i$ but have no conditions imposed on $\Sigma$. We then define the space of two-sided Cauchy data on the separating hypersurface $\Sigma$ by
\begin{align}\label{mudef}
	\mu(\lambda) = \left\{ \left.\left(u_1,\frac{\p u_1}{\p \nu_1},u_2,-\frac{\p u_2}{\p \nu_2} \right)\right|_\Sigma : u_i \in K_i^\lambda \right\}.
\end{align}

We also consider the one-parameter family of boundary conditions on $\Sigma$, given by
\begin{align}\label{betadef}
	\beta(t) = \{(x,t\phi,tx,\phi) : x \in \Hp, \phi \in \Hm \}
\end{align}
for $t\in[0,1]$. It will be shown that $\mu(\lambda)$ and $\beta(t)$ comprise smooth families of Lagrangian subspaces in $\cH_\boxplus$, and form a Fredholm pair for every $(\lambda,t) \in \bbR \times [0,1]$, so one can define the \textit{Maslov index} of $\beta(t)$ with respect to $\mu(\lambda_0)$ for any fixed $\lambda_0$. This is a homotopy invariant quantity that counts the intersections of the subspaces $\beta(t)$ and $\mu(\lambda_0)$, with sign and multiplicity, as $t$ increases from 0 to 1.

We are now ready to state the main result of the paper.

\begin{theorem}\label{MorMas}
Let $(M,g)$ be a Riemannian manifold with Lipschitz boundary $\p M$ and a Lipschitz separating hypersurface $\Sigma$. The operators $L^D_i$, $L^N_i$ and $L^G$ defined in \eqref{LDdef}, \eqref{LNdef} and \eqref{LGdef} satisfy
\begin{align}\label{GDNequality}
	\Mor(L^G) = \Mor(L^N_1) + \Mor(L^D_2) + \Mas(\beta(t); \mu(0)).
\end{align}
\end{theorem}

In practice we can view this as a tool for computing the spectrum of $L^G$ by decomposing $M$ into two simpler pieces, $\Omega_1$ and $\Omega_2$. The relation between these spectral problems is given by the index $\Mas(\beta(t); \mu(0))$, so it remains to understand this term. We take several approaches to this. The first is to use spectral properties of the Dirichlet-to-Neumann maps for $\Omega_1$ and $\Omega_2$, denoted $\Lambda_1$ and $\Lambda_2$, respectively, to compute the number of positive and negative intersections the path $\beta(t)$ has with the fixed subspace $\mu(0)$.

\begin{theorem}\label{DNcrossing}
If the hypotheses of Theorem \ref{MorMas} are satisfied and $0 \notin \sigma(L^D_1) \cup \sigma(L^D_2)$, then
\[
	\Mas(\beta(t); \mu(0)) = \Mor_0(\Lambda_1 + \Lambda_2) - \Mor_0(\Lambda_1).
\]
\end{theorem}

This is particularly useful when combined with the following well-known result on the Dirichlet-to-Neumann map.

\begin{theorem}[Friedlander \cite{F91}]\label{thmFriedlander}
Fix $i \in \{1,2\}$. If $0 \notin \sigma(L^D_i)$, then
\[
	\Mor(L^N_i) - \Mor(L^D_i) = \Mor_0(\Lambda_i).
\]
\end{theorem}
This result was first proved by Friedlander in \cite{F91}, where it was used to establish an inequality between the Dirichlet and Neumann eigenvalues of a Euclidean domain. In \cite{M91} Mazzeo gave a more geometric proof and considered the possibility of generalizing Friedlander's approach to non-Euclidean geometries. In Section \ref{secDN} we discuss Mazzeo's proof in a symplectic framework.

From Theorems \ref{MorMas}, \ref{DNcrossing} and \ref{thmFriedlander} we obtain the following. 

\begin{cor}\label{DNbracket}
Assume the hypotheses of Theorem \ref{DNcrossing}. Then
\begin{align}
	\Mor(L^G) = \Mor(L^D_1) + \Mor(L^D_2) + \Mor_0(\Lambda_1 + \Lambda_2),
\end{align}
hence
\begin{align}\label{GDNinequality}
	\Mor(L^D_1) + \Mor(L^D_2) \leq \Mor(L^G) \leq \Mor(L^N_1) + \Mor(L^N_2).
\end{align}
\end{cor}

It is well known that \eqref{GDNinequality} can be derived by a min-max argument---see Proposition XIII.15.4 of \cite{RS78}, where it is observed that the operator $L$ increases when one either adds a hypersurface with Dirichlet boundary conditions or removes a hypersurface with Neumann boundary conditions. Theorem \ref{MorMas} is a quantitative improvement of this result, since it provides the additional information that the inequality
\[
	\Mor(L^G) \geq \Mor(L^D_1) + \Mor(L^D_2)
\]
is strict when $\Mor(\Lambda_1 + \Lambda_2) > 0$, and the inequality
\[
	\Mor(L^G) \leq \Mor(L^N_1) + \Mor(L^N_2)
\]
is strict when $\Mor_0(\Lambda_1 + \Lambda_2) < \Mor_0(\Lambda_1) + \Mor_0(\Lambda_1)$.

The $n$-sphere yields a simple example in which both parts of \eqref{GDNinequality} are strict. Define $L = -\Delta_g - c$  for $c \in \bbR$, where $\Delta_g$ is the Laplace--Beltrami operator on $\bbS^n$, and let $L^G$ denote the global realization of $L$, and $L^D$, $L^N$ the Dirichlet and Neumann realizations of $L$ on the upper hemisphere $\bbS^n_+$. It follows from a reflection argument (cf. Theorem \ref{perturb}) that $\Mor(L^G) = \Mor(L^D) + \Mor(L^N)$, whereas \eqref{GDNinequality} yields $2\Mor(L^D) \leq \Mor(L^G) \leq 2 \Mor(L^N)$. Therefore both inequalities are strict when $\Mor(L^D) < \Mor(L^N)$. This can be achieved by choosing $c > 0$ smaller than the first Dirichlet eigenvalue of $-\Delta_g$ on $\bbS^n_+$.

In general the question of when $\Mor(L^D) < \Mor(L^N)$ is quite subtle, and is intimately related to inequalities between the Dirichlet and Neumann eigenvalues---see \cite{F91,M91} and references therein.

An immediate application of Corollary \ref{DNbracket} is to the study of nodal domains. Suppose $\phi_k$ is the $k$th eigenfunction of $L$, with eigenvalue $\lambda_k$. The \textit{nodal domains} of $\phi_k$ are the connected components of the set $\{\phi_k \neq 0\}$. We denote the total number of nodal domains by $n(\phi_k)$, and define the \textit{nodal deficiency}
\begin{align}
	\delta(\phi_k) = k - n(\phi_k).
\end{align}

\begin{cor}\label{nodalDef}
Suppose $\lambda_k$ is a simple eigenvalue of $L$ and $0$ is a regular value of $\phi_k$. Define $L(\epsilon) = L - (\lambda_k + \epsilon)$ and let $\Lambda_\pm(\epsilon)$ denote the corresponding Dirichlet-to-Neumann maps on $\Omega_\pm = \{\pm\phi_k > 0\}$. If $\epsilon > 0$ is sufficiently small, then
\[
	\delta(\phi_k) = \Mor\left( \Lambda_+(\epsilon) + \Lambda_-(\epsilon) \right).
\]
\end{cor}

Since the right-hand side is nonnegative, this implies $n(\phi_k) \leq k$, which is Courant's nodal domain theorem \cite{CH53}. In \cite{BKS12} Berkolaiko, Kuchment and Smilinsky gave a different formula for the nodal deficiency as the Morse index of a certain energy functional defined on the space of equipartitions of $M$.

In some cases we can compute $\Mas(\beta(t); \mu(0))$ by finding conditions that ensure the index vanishes. The easiest case is when the Dirichlet-to-Neumann maps for $\Omega_1$ and $\Omega_2$ coincide. In this situation Theorem \ref{DNcrossing} yields
\[
	\Mas(\beta(t); \mu(0)) =  \Mor_0(\Lambda_1 + \Lambda_2) - \Mor_0(\Lambda_1) = 0
\]
because $\Lambda_1 + \Lambda_2 = 2\Lambda_1$ has the same Morse index as $\Lambda_1$.

The problem of determining when $\Lambda_1 = \Lambda_2$ is in general quite difficult, even for $L = -\Delta_g$. If $\tau\colon M \to M$ is an isometry with $\tau(\Sigma_D) = \Sigma_D$, $\tau(\Sigma_N) = \tau(\Sigma_N)$ and $\left.\tau\right|_\Sigma = \operatorname{id}$, then $\Lambda_1 = \Lambda_2$. In the real analytic case, $\Lambda_1 = \Lambda_2$ implies $\Omega_1$ and $\Omega_2$ are isometric \cite{LTU03}, but the smooth case is still unresolved.

However, we are able to prove that $\beta(t) \cap \mu(0) = \{0\}$ for all $t$, hence $\Mas(\beta(t); \mu(0)) = 0$, provided the Dirichlet-to-Neumann maps are sufficiently close. To that end, it is convenient to view them as bounded operators
\[
	\tL_i \colon \Hp \lra \Hm,
\]
which are well defined if $0 \notin \sigma(L^D_i)$. If $0 \notin \sigma(L^N_i)$, by an abuse of notation we let
\[
	\tL_i^{-1} \colon \Hm \lra \Hp
\]
denote the corresponding Neumann-to-Dirichlet map. (If $0 \notin \sigma(L^D_i) \cup \sigma(L^N_i)$, the operators $\tL_i$ and $\tL_i^{-1}$ both exist and are mutually inverse.)

\begin{theorem}\label{perturb}
Assume $0 \notin \sigma(L^N_1) \cup \sigma(L^D_2)$. If there exists $c$ such that
\[
	\left\| \tL_1^{-1} \tL_2 - cI \right\|_{B(\Hp)} < 1+c,
\]
then $\Mas(\beta(t); \mu(0)) = 0$, hence
\[
	\Mor(L^G) = \Mor(L^N_1) + \Mor(L^D_2).
\]
\end{theorem}

\begin{center}
\begin{figure}
\begin{tikzpicture}
	\draw[thick] (0,0) -- (4,0); 
	\draw[thick] (0,2) -- (4,2); 
	\draw[thick,dashed] (0,0) arc (-90:90:0.2 and 1); 
	\draw[thick] (0,0) arc (270:90:0.2 and 1); 
	\draw[thick,dashed] (2,0) arc (-90:90:0.2 and 1); 
	\draw[thick] (2,0) arc (270:90:0.2 and 1); 
	\draw[thick] (4,1) ellipse (0.2 and 1); 
	\node at (0.8,1) {$\Omega_1$}; 
	\node at (1.5,1) {$\Sigma$}; 
	\node at (3,1) {$\Omega_2$}; 
	\node at (-0.6,1) {$\p M$}; 
	\node at (4.6,1) {$\p M$}; 
\end{tikzpicture}
\caption{An example of the doubling construction in Section \ref{secDoubled}.}
\label{double}
\end{figure}
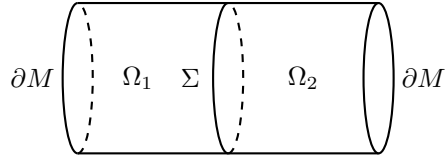
\end{center}

The corollary applies to the cylinder shown in Figure \ref{double} if one prescribes the same boundary conditions (either Dirichlet or Neumann) on both ends of the cylinder. With mixed boundary conditions it is possible that the Maslov index is nonzero, as is demonstrated by a simple example in Section \ref{secDoubled}.

Using similar methods, we can also describe the spectra of periodic eigenvalue problems and, more generally, problems in which the boundary is divided into two components, $\pM = \Gamma_1 \cup \Gamma_2$, which are identified by a map $\tau\colon \Gamma_1 \to \Gamma_2$ as shown in Figure \ref{periodic}. We let $L^D$ denote the Dirichlet realization of $L$, and $L^P$ the ``periodic" realization, with domain
\[
	\cD(L^P) = \left\{u \in H^1(M) : Lu \in L^2(M), \left.u\right|_{\Gamma_1} = \left.u\right|_{\Gamma_2} \circ \tau
	\text{ and } \left.\frac{\p u}{\p \nu}\right|_{\Gamma_1} = -\left.\frac{\p u}{\p \nu}\right|_{\Gamma_2} \circ \tau \right\}.
\]
To state the result we must also define the ``periodic Dirichlet-to-Neumann map" $\Lambda_\tau$, assuming $0 \notin \sigma(L^D)$. For a function $f$ on $\Gamma_1$ we let $u$ denote the unique solution to the boundary value problem
\[
	Lu = 0, \quad \left.u\right|_{\Gamma_1} = f, \quad \left.u\right|_{\Gamma_2} = f \circ \tau^{-1}.
\]
and define
\begin{align}
\label{Ltdef}
	\Lambda_\tau f = \left.\frac{\p u}{\p \nu}\right|_{\Gamma_1} + \left.\frac{\p u}{\p \nu} \right|_{\Gamma_2}\circ \tau.
\end{align}
It is shown in Section \ref{secPeriodic} that $\Lambda_\tau$ defines an unbounded, selfadjoint operator on $L^2(\Gamma_1)$, with domain
\[
	\cD(\Lambda_\tau) = \left\{ f \in L^2(\Gamma_1) : \left.\frac{\p u}{\p \nu}\right|_{\Gamma_1} + \left.\frac{\p u}{\p \nu} \right|_{\Gamma_2}\circ \tau \in L^2(\Gamma_1) \right\}.
\]
Moreover, it is bounded from below and has compact resolvent, and hence has a well-defined Morse index.

For the following theorem to hold, it is necessary that the boundary can be subdivided into pieces on which the map $\tau$ is Lipschitz. (In general $\tau$ will not be globally Lipschitz---for the cube with opposing faces identified it fails to be continuous at the corners.) We let $d\mu_1$ and $d\mu_2$ denote the induced area forms on $\Gamma_1$ and $\Gamma_2$, respectively.

\begin{center}
\begin{figure}
\begin{tikzpicture}
	\draw[thick] (0,0) -- (0,3); 
	\draw[thick] (0,0) -- (3,0); 
	\draw[very thick,dashed] (0,3) -- (3,3); 
	\draw[very thick, dashed] (3,0) -- (3,3); 
	\node at (1.5,-0.3) {$\Gamma_1$}; 
	\node at (1.5,3.3) {$\Gamma_2$}; 
	\node at (-0.3,1.5) {$\Gamma_1$}; 
	\node at (3.3,1.5) {$\Gamma_2$}; 
	\node at (1.5,1.5) {$M$}; 
\end{tikzpicture}
\caption{A manifold $M$ with boundary $\pM = \Gamma_1 \cup \Gamma_2$, as in the statement of Theorem \ref{thmPeriodic}.}
\label{periodic}
\end{figure}
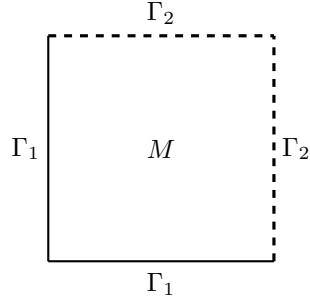
\end{center}

\begin{theorem}\label{thmPeriodic}
Suppose $\Gamma_1$ can be decomposed as $\overline{\Gamma_1^1 \cup \cdots \cup \Gamma_1^N}$, where each $\Gamma_1^i$ is an open subset of $\pM$ with Lipschitz boundary, and the restrictions $\left.\tau\right|_{\Gamma_1^i} \colon\Gamma_1^i \to \tau(\Gamma_1^i)$ are Lipschitz. If $0 \notin \sigma(L^D)$ and $\tau^* d\mu_2 = d\mu_1$, then
\[
	\Mor(L^P) = \Mor(L^D) + \Mor_0(\Lambda_\tau).
\]
\end{theorem}

Thus the periodic problem on $M$ is related to the Dirichlet problem, and the difference in Morse indices is quantified by the periodic Dirichlet-to-Neumann map $\Lambda_\tau$. This is useful because separated boundary conditions are often easier to work with than periodic boundary conditions and more techniques are available for their study.


\section{Proof of the main theorem}\label{secProof}
In this section we prove Theorem \ref{MorMas}. We first describe how the subspaces $\mu(\lambda)$ and $\beta(t)$, defined in \eqref{mudef} and \eqref{betadef}, contain spectral data for the operators $L^G$, $L^D_i$ and $L^N_i$. Next we prove that $\mu(\lambda)$ and $\beta(t)$ are smooth families of Lagrangian subspaces in $\cH_\boxplus$ (as defined in \eqref{Hbpdef}) and comprise a Fredholm pair, so their Maslov index is well defined. Finally, we use the homotopy invariance of the Maslov index to prove the theorem.

Recall that $\mu(\lambda)$ encodes the boundary data of weak solutions to the equation $Lu = \lambda u$, with no boundary conditions imposed on $\Sigma$, whereas $\beta(t)$ defines a one-parameter family of boundary conditions that does not depend on $L$. The significance of the endpoints $t=0,1$ is the following.

\begin{lemma}\label{intersection}
 Let $\lambda \in \bbR$. Then
 \begin{align*}
 	\dim \left[\mu(\lambda) \cap \beta(0)\right] &= \dim \ker (L^N_1 - \lambda) + \dim \ker (L^D_2 - \lambda) \\
	\dim \left[\mu(\lambda) \cap \beta(1)\right] &= \dim \ker (L^G - \lambda).
 \end{align*}
\end{lemma}

The proof relies on a version of the unique continuation principle (Proposition 2.5 of \cite{BR12}) which says that there is a one-to-one correspondence between weak solutions in $K^\lambda_i$ and their Cauchy data in $\Hp \oplus \Hm$; cf. Proposition 2.2 in \cite{CJM14}.

\begin{proof} 
For the first claim observe that $\beta(0) = \{(x,0,0,\phi) : x \in \Hp, \phi \in \Hm \}$, and so $\mu(\lambda) \cap \beta(0)$ is nontrivial when there exist functions $u_i \in K_i^\lambda$, not both zero, such that
\[
	\left. \frac{\p u_1}{\p \nu_1} \right|_{\Sigma} = 0, \quad \left. u_2 \right|_{\Sigma} = 0.
\]
Therefore $u_1 \in \cD(L^N_1)$ and $u_2 \in \cD(L^D_2)$, with $(L^N_1 - \lambda) u_1 = 0$ and $(L^D_2 - \lambda)u_2 = 0$. Since (at least) one of $u_1$ and $u_2$ is nonzero, we conclude that $\lambda \in \sigma(L^N_1) \cup \sigma(L^D_2)$. On the other hand, if $\lambda \in \sigma(L^N_1)$, the corresponding eigenfunction satisfies $u_1\in K_1^\lambda$, hence
\[
	\left( \left.u_1\right|_\Sigma,0,0, 0 \right) \in \mu(\lambda) \cap \beta(0),
\]
and similarly when $\lambda \in \sigma(L^D_2)$. This completes the proof of the first equality.

For the second equality we observe that $\beta(1) = \{(x,\phi,x,\phi) : x \in \Hp, \phi \in \Hm \}$, and so $\mu(\lambda) \cap \beta(1)$ is nontrivial if and only if there exist functions $u_i \in K_i^\lambda$ such that
\[
	\left. u_1 \right|_{\Sigma} = \left. u_2 \right|_{\Sigma}, \quad \left. \frac{\p u_1}{\p \nu_1} \right|_{\Sigma} = -\left. \frac{\p u_2}{\p \nu_2} \right|_{\Sigma}.
\]
But this is true precisely when there exists a weak solution $u \in H^1(M)$ to $Lu = \lambda u$ (with $\left.u\right|_{\Omega_i} = u_i$ for $i \in \{1,2\}$) such that
\[
	\left.u\right|_{\Sigma_D} = 0,\quad \left.\frac{\p u}{\p \nu}\right|_{\Sigma_N} = 0,
\]
which is equivalent to $\lambda \in \sigma(L^G)$.
%
%
\end{proof}

We next show that the set of $\lambda$ for which $\mu(\lambda)$ and  $\beta(t)$ intersect nontrivially is bounded below uniformly in $t$.

\begin{lemma}\label{uniformbound}
There exists $\lambda_\infty < 0$ such that $\mu(\lambda) \cap \beta(t) = \{0\}$ for all $\lambda \leq \lambda_\infty$ and $t \in [0,1]$.
\end{lemma}

\begin{proof} Suppose $\mu(\lambda) \cap \beta(t) \neq \{0\}$. By definition, there exist functions $u_i \in K_i^\lambda$ such that
\[
	\left.u_2\right|_\Sigma = t \left.u_1\right|_\Sigma, \quad \left. \frac{\p u_1}{\p \nu_1}\right|_\Sigma = -t \left.\frac{\p u_2}{\p \nu_2}\right|_\Sigma,
\]
hence
\[
	\left. u_1 \frac{\p u_1}{\p \nu_1}\right|_\Sigma = - \left. u_2 \frac{\p u_2}{\p \nu_2}\right|_\Sigma.
\]
Integrating by parts, we obtain
\begin{align*}
	\int_{\Omega_1} \left[|\nabla u_1|^2 + (V - \lambda)u_1^2\right]
	& = \int_\Sigma u_1 \frac{\p u_1}{\p \nu_1} \\
	&= -\int_\Sigma u_2 \frac{\p u_2}{\p \nu_2} 
	= -\int_{\Omega_2} \left[ |\nabla u_2|^2 + (V - \lambda)u_2^2 \right]
\end{align*}
which implies
\[
	\lambda \left( \int_{\Omega_1} u_1^2 +  \int_{\Omega_2} u_2^2 \right) \geq \int_{\Omega_1} V u_1^2 + \int_{\Omega_2}V u_2^2
\]
and so it suffices to choose
\[
	\lambda_\infty < \inf_{x\in M} V(x).
\]
\end{proof}

To define the Maslov index of $\mu$ with respect to $\beta$ (and vice versa), we need to prove that $\mu(\lambda)$ and $\beta(t)$ are continuous curves in the Lagrangian Grassmannian of $\cH_\boxplus$ and comprise a Fredholm pair for each $\lambda$ and $t$. We recall that a curve $\gamma\colon I \to \Lambda(\cH_\boxplus)$ is said to be $C^k$ if the corresponding curve of orthogonal projections, $t \mapsto P_{\gamma(t)}$, is contained in $C^k\left(I, B(\cH_\boxplus)\right)$.

\begin{lemma} For $(\lambda,t) \in \bbR \times [0,1]$ the subspaces $\mu(\lambda)$ and $\beta(t)$ are Lagrangian, and the curves $\lambda \mapsto \mu(\lambda)$ and $t \mapsto \beta(t)$ are smooth.
\end{lemma}

\begin{proof}
The space $\mu(\lambda)$ of two-sided Cauchy data can be decomposed as $\mu(\lambda) = \mu_1(\lambda) \oplus \mu_2(\lambda)$, where $\mu_1$ and $\mu_2$ are the spaces of Cauchy data for weak solutions to $Lu = \lambda u$ on $\Omega_1$ and $\Omega_2$, respectively. It was shown in Proposition 3.5 of \cite{CJM14} that $\lambda \mapsto \mu_1(\lambda)$ and $\lambda \mapsto \mu_2(\lambda)$ are smooth curves in $\Lambda(\cH)$, hence their sum $\mu(\lambda)$ is a smooth curve in $\Lambda(\cH_\boxplus)$.

The fact that $\beta(t)$ is Lagrangian follows from a direct computation, and the regularity of $t \mapsto \beta(t)$ is immediate from the definition.
\end{proof}

\begin{lemma}\label{propFredholm}
For $(\lambda,t) \in \bbR \times [0,1]$, $\mu(\lambda)$ and $\beta(t)$ comprise a Fredholm pair.
\end{lemma}

\begin{proof}
For convenience we abbreviate $\mu = \mu(\lambda)$ and $\beta = \beta(t)$. Let $P_\beta\colon \cH_\boxplus \to \cH_\boxplus$ denote the orthogonal projection onto $\beta$, and $P_\beta^\bot = I - P_\beta$ the complementary projection. By Proposition 2.27 of \cite{F04}, $\mu$ and $\beta$ comprise a Fredholm pair if and only if the restriction $P_\beta^\bot \big|_\mu \colon \mu \to \cH_\boxplus$ is a Fredholm operator.

By Peetre's lemma (Lemma 3 of \cite{P61}), it suffices to find a compact embedding $\iota \colon \mu \to Y$ and a positive constant $C$ such that
\[
	\|z\|_{\cH_\boxplus} \leq C \left( \|P_\beta^\bot z\|_{\cH_\boxplus} + \|\iota z\|_Y \right)
\]
for all $z \in \mu$. Since the operator $\tr\colon K_1^\lambda \oplus K_2^\lambda \to \mu$ defined by
\[
	\tr (u_1,u_2) = \left.\left(u_1,\frac{\p u_1}{\p \nu_1},u_2,-\frac{\p u_2}{\p \nu_2} \right)\right|_\Sigma
\]
is boundedly invertible by Lemmas 3.2 and 3.3 of \cite{CJM14}, it suffices to prove
\[
	\|(u_1,u_2)\|_{H^1} \leq C \left( \|P_\beta^\bot \tr(u_1,u_2) \|_{\cH_\boxplus} + \|\iota \tr(u_1,u_2)\|_Y \right)
\]
for $u_i \in K_i^\lambda$, where we have defined $\|(u_1,u_2)\|_{H^1}^2 = \|u_1\|_{H^1(\Omega_1)}^2 + \|u_2\|_{H^1(\Omega_2)}^2$.
Defining $Y = L^2(\Omega_1) \oplus L^2(\Omega_2)$ and letting $\iota\colon \mu \to Y$ denote the composition
\[
	\mu \xrightarrow{\tr^{-1}} H^1(\Omega_1) \oplus H^1(\Omega_2) \longhookrightarrow L^2(\Omega_1) \oplus L^2(\Omega_2),
\]
we need to show that
\begin{align}\label{energy}
	\|(u_1,u_2)\|_{H^1} \leq C \left( \|P_\beta^\bot \tr(u_1,u_2) \|_{\cH_\boxplus} + \|(u_1,u_2)\|_{L^2} \right)
\end{align}
for $u_i \in K_i^\lambda$.

To prove \eqref{energy}, we first observe that there is a constant $C > 0$ such that
\begin{align}\label{energy1}
	\|u_i\|_{H^1(\Omega_i)}^2 \leq C \|u_i\|_{L^2(\Omega_i)}^2 + \int_\Sigma u_i \frac{\p u_i}{\p \nu_i}
\end{align}
for $i \in \{1,2\}$ and $u_i \in K^\lambda_i$. Using the definition of $\beta = \beta(t)$ we can write
\begin{align*}
	P_\beta \tr (u_1,u_2) = (x, t\phi, tx, \phi), \quad  P_\beta^\bot \tr (u_1,u_2) = (-ty, \psi, y, -t\psi)
\end{align*}
for some $x,y \in \Hp$ and $\phi, \psi \in \Hm$. 
Therefore
\begin{align*}
	\int_\Sigma u_1 \frac{\p u_1}{\p \nu_1} + u_2 \frac{\p u_2}{\p \nu_2} &= (t\phi + \psi)(x-ty) - (\phi - t\psi)(tx+y) \\
	&= (1+t^2) \left[ \psi(x) - \phi(y)\right] \\
	& \leq \frac{1+t^2}{2} \left( \epsilon \|x\|_{\Hp} + \epsilon^{-1} \|\psi\|_{\Hm} + \epsilon^{-1} \|y\|_{\Hp} + \epsilon \|\phi\|_{\Hm} \right) \\
	&= \frac{\epsilon}{2} \| P_\beta \tr (u_1,u_2) \|_{\cH_\boxplus}^2 + \frac{1}{2\epsilon} \| P_\beta^\bot \tr (u_1,u_2) \|_{\cH_\boxplus}^2
\end{align*}
for any $\epsilon > 0$, hence \eqref{energy1} implies
\[
	\|(u_1,u_2)\|_{H^1}^2 \leq C \|(u_1,u_2)\|_{L^2}^2 + \frac{\epsilon}{2} \| P_\beta \tr (u_1,u_2) \|_{\cH_\boxplus}^2 + \frac{1}{2\epsilon} \| P_\beta^\bot \tr (u_1,u_2) \|_{\cH_\boxplus}^2.
\]
Choosing $\epsilon$ small enough that $\epsilon \| P_\beta \tr (u_1,u_2) \|_{\cH_\boxplus}^2 \leq \|(u_1,u_2)\|_{H^1}$, which is possible by Lemma 3.2 of \cite{CJM14}, the desired estimate \eqref{energy} follows.
\end{proof}

\begin{rem}
The proof of Lemma 3.8 in \cite{CJM14}, in which certain pairs of Lagrangian subspaces are shown to be Fredholm, can be greatly simplified by an application Peetre's lemma as above.
\end{rem}

%
%

The Maslov index counts signed intersections of Lagrangian subspaces and so, in light of Lemma \ref{intersection}, it is not surprising that the Maslov indices of $\mu(\lambda)$ with respect to $\beta(0)$ and $\beta(1)$ are related to the Morse indices of the corresponding boundary value problems. This is a consequence of the fact that $\mu(\lambda)$ always passes through $\beta(t_0)$ in the same direction. This is proved by computing the crossing form---a symmetric bilinear form associated to a nontrivial intersection---and showing that it is sign definite (c.f. the proof of Lemma 4.2 in \cite{CJM14}). The necessary properties of crossing forms can be found in Appendix B of \cite{CJM14}; see also \cite{F04} for a more thorough treatment.

\begin{prop} \label{lambdaMaslov}
For $|\lambda_\infty|$ sufficiently large we have
\[
	\Mas(\mu(\lambda); \beta(0)) = -\Mor(L^N_1) - \Mor(L^D_2)
\]
and
\[
	\Mas(\mu(\lambda); \beta(1)) = -\Mor(L^G),
\]
where the Maslov index is computed over the interval $[\lambda_\infty,0]$.
\end{prop}

\begin{proof} We use a crossing form computation to show that, for fixed $t_0 \in [0,1]$, every intersection of $\mu(\lambda)$ with $\beta(t_0)$ is negative definite. This implies
\begin{align*}
	\Mas(\mu(\lambda); \beta(t_0)) &= - \sum_{\lambda_\infty \leq \lambda < 0} \dim \left( \mu(\lambda) \cap \beta(t_0)\right) \\
	&=  -\sum_{\lambda < 0} \dim \left( \mu(\lambda) \cap \beta(t_0)\right),
\end{align*}
where in the second equality we have used Lemma \ref{uniformbound}, and the result then follows from Lemma \ref{intersection}.

To prove monotonicity, we assume there is a crossing at some $\lambda_* \in \bbR$. Then there exist differentiable paths of functions $\lambda \mapsto u_i(\lambda) \in H^1(\Omega_i)$ such that $u_i(\lambda) \in K_i^\lambda$ for $|\lambda - \lambda_*| \ll 1$, hence
\[
	z(\lambda) = \left.\left(u_1(\lambda), \frac{\p u_1}{\p \nu_1}(\lambda), u_2(\lambda), -\frac{\p u_2}{\p \nu_2}(\lambda) \right)\right|_\Sigma
\]
defines a differentiable curve in $\cH_\boxplus$, with $z(\lambda) \in \mu(\lambda)$ and $z(\lambda_*) \in \mu(\lambda_*) \cap \beta(t_0)$. Since $\omega_\boxplus = \omega \oplus (-\omega)$, the crossing form is given by
\begin{align}\label{crossing12}
	\omega_\boxplus \left(z, \frac{dz}{d\lambda}\right) &= \omega \left(z_1, \frac{dz_1}{d\lambda}\right) - \omega \left(z_2, \frac{dz_2}{d\lambda}\right).
\end{align}

To compute the first term on the right-hand side of \eqref{crossing12}, we define the quadratic form
\[
	\Phi(u,v) = \int_{\Omega_1} \left[ g(\nabla u, \nabla v) + Vuv \right]
\]
for $u,v \in H^1(\Omega_1)$. Since $u_1(\lambda) \in K_1^\lambda$, Green's first identity \eqref{Green} implies
\begin{align}\label{PhiGreen}
	\Phi(u_1(\lambda),v) = \lambda \left< u_1(\lambda),v\right>_{L^2(\Omega_1)} + \int_{\Sigma} v \frac{\p u_1}{\p \nu_1}
\end{align}
for any $v \in H^1(\Omega_1)$ with $\left.v\right|_{\Sigma_D \cap \p \Omega_1} = 0$. Choosing $v = du_1/d\lambda$, we obtain
\[
	\Phi\left(u_1,\frac{du_1}{d\lambda} \right) = \lambda \left< u_1,\frac{du_1}{d\lambda} \right>_{L^2(\Omega_1)} + \int_{\Sigma} \frac{du_1}{d\lambda}  \frac{\p u_1}{\p \nu_1}.
\]
On the other hand, differentiating \eqref{PhiGreen} with respect to $\lambda$ and then choosing $v = u_1$ yields
\[
	\Phi\left(\frac{du_1}{d\lambda}, u_1 \right) = \| u_1\|_{L^2(\Omega_1)}^2  + \lambda \left< \frac{du_1}{d\lambda}, u_1 \right>_{L^2(\Omega_1)} + \int_{\Sigma} u_1 \frac{d}{d\lambda} \frac{\p u_1}{\p \nu_1}.
\]
Using the symmetry of $\Phi$ and recalling the definition of $\omega$, it follows that
\[
	\left.\omega \left(z_1, \frac{dz_1}{d\lambda}\right)\right|_{\lambda = \lambda_*} = \int_{\Sigma} \left[ u_1 \frac{d}{d\lambda} \frac{\p u_1}{\p \nu_1} - \frac{du}{d\lambda_1} \frac{\p u_1}{\p \nu_1} \right] 
	= - \| u_1(\lambda_*)\|_{L^2(\Omega_1)}^2.
\]

The second term on the right-hand side of \eqref{crossing12} is computed similarly, and we obtain for the crossing form
\[
	\left.\omega_\boxplus \left(z, \frac{dz}{d\lambda}\right)\right|_{\lambda = \lambda_*} = 
	- \| u_1(\lambda_*) \|_{L^2(\Omega_1)}^2 - \| u_2(\lambda_*) \|_{L^2(\Omega_2)}^2,
\]
which is strictly negative.
\end{proof}

We are now ready to prove Theorem \ref{MorMas}.

\begin{proof}[Proof of Theorem \ref{MorMas}]

The unitary group acts transitively on $\Lambda(\cH_\boxplus)$ and, by Theorem 2.14 of \cite{F04}, gives it the structure of a principal fiber bundle, so there exists a continuous family of unitary operators $U(t)\colon \cH_\boxplus \to \cH_\boxplus$ such that $\beta(t)= U(t) \beta(0)$. 
 We define a homotopy $[\lambda_\infty,0] \times [0,1] \to \mathcal{F} \Lambda_{\beta(0)} (\cH_\boxplus)$ by $(\lambda,t) \mapsto U(t)^{-1}\mu(\lambda)$. The invariance of the Maslov index under unitary transformations implies
\[
	\Mas \left(U(t_0)^{-1}\mu(\lambda); \beta(0) \right) = \Mas(\mu(\lambda); \beta(t_0))
\]
and
\[
	\Mas \left(U(t)^{-1}\mu(\lambda_0); \beta(0) \right) = -\Mas(\beta(t); \mu(\lambda_0))
\]
for any fixed $t_0$ and $\lambda_0$. The image of the boundary of $[\lambda_\infty,0] \times [0,1]$ is null homotopic in $\mathcal{F} \Lambda_{\beta(0)} (\cH_\boxplus)$ and hence has zero Maslov index. This implies
\[
	\Mas(\mu(\lambda); \beta(1)) = \Mas(\mu(\lambda);\beta(0)) + \Mas(\beta(t);\mu(\lambda_\infty)) - \Mas(\beta(t);\mu(0)).
\]
The proof follows immediately from the above formula, Lemma \ref{uniformbound} (which implies $\Mas(\beta(t);\mu(\lambda_\infty)) = 0$) and Proposition \ref{lambdaMaslov}.

\end{proof}


\section{The Maslov index of $\beta(t)$}\label{secCrossing}
In this section we prove Theorem \ref{DNcrossing}: if $0 \notin \sigma(L^D_1) \cup \sigma(L^D_2)$, then
\[
	\Mas(\beta(t); \mu(0)) = \Mor_0(\Lambda_1 + \Lambda_2) - \Mor_0(\Lambda_1).
\]
Instead of directly analyzing the crossings of $\beta(t)$ with $\mu(0)$, which may be degenerate, we deform $\beta(t)$ to a nondegenerate path for which the Maslov index can be easily computed, and then appeal to the homotopy invariance of the index.

We first define the (unbounded) Dirichlet-to-Neumann map $\Lambda_i$ for $i \in \{1,2\}$. This is an unbounded operator on $L^2(\Sigma)$ with domain
\[
	\cD(\Lambda_i) = \left\{f \in L^2(\Sigma) : \exists \, u \in K^0_i \text{ such that } \left. u\right|_\Sigma = f \text{ and } \left. \frac{\p u}{\p \nu_i}\right|_\Sigma \in L^2(\Sigma) \right\},
\]
defined by $\Lambda_i f = \left. \frac{\p u}{\p \nu_i}\right|_\Sigma$. Recall from \eqref{Kdef} that $u \in K^0_i$ means that
\[
	\left. u\right|_{\Sigma_D \cap \pO_i} = 0, \quad  \left. \frac{\p u}{\p \nu}\right|_{\Sigma_N \cap \pO_i} = 0
\]
and $u$ is a weak solution to the equation $Lu = 0$ in $\Omega_i$.

We can also view the Dirichlet-to-Neumann maps as \emph{bounded} operators
\[
	\tL_1, \tL_2 \colon \Hp \longrightarrow \Hm.
\]
We relate the spectrum of the unbounded operator $\Lambda_i$ to its bounded counterpart $\tL_i$. Let $\cJ \colon\Hp\hookrightarrow\Hm$ denote the compact inclusion. 

%

\begin{lemma}\label{DNdomains}
Let $i \in \{1,2\}$ and suppose $0 \notin \sigma(L^D_i)$. Then $s \in \sigma(\Lambda_i)$ if and only if there exists $f \in \Hp$ such that $(\tL_i -s \cJ) f = 0$.
\end{lemma}

\begin{proof}
First suppose that $s \in \sigma(\Lambda_i)$, so there exists $u_i \in K_i^0$ with
\[
	\left. \frac{\p u_i}{\p \nu_i}\right|_\Sigma = s \left. u_i \right|_\Sigma \in L^2(\Sigma).
\]
Then $f := \left. u_i \right|_\Sigma$ is contained in $\Hp$ and satisfies $\tL_i f = s \mathcal{J} f$, as required.

On the other hand, suppose $u_i \in K_i^0$ satisfies
\[
	\left. \frac{\p u_i}{\p \nu_i}\right|_\Sigma = s \mathcal{J} \left(\left. u_i \right|_\Sigma\right) \in \Hm.
\]
This implies $\left. \frac{\p u_i}{\p \nu_i}\right|_\Sigma \in \Hp$, hence $f := \left. u_i \right|_\Sigma$ is contained in $\cD(\Lambda_i)$ and satisfies $\Lambda_i f = sf$.
\end{proof}

%

For $s \leq 0$ we consider the two-parameter family of boundary data
\begin{align}\label{betast}
	\beta(s,t) = \left\{ (x, t\phi + s \mathcal{J} x, tx, \phi) : x \in \Hp, \phi \in \Hm \right\}.	
\end{align}
When $s=0$ this is just $\beta(t)$. It is easy to see that
\[
	\beta(s,t)^\bot = \left\{ (-ty - s \mathcal{J}^*\psi, \psi, y, -t\psi) : y \in \Hp, \psi \in \Hm \right\}.	
\]
This modification of $\beta$ yields a homotopy that relates the Maslov index of $\beta(t)$ to the Morse indices of the Dirchlet-to-Neumann maps.  This relies crucially on a certain monotonicity with respect to $s$, which is shown in Proposition \ref{betamorse}.

\begin{lemma}\label{betaLag}
If $(s,t) \in (-\infty,0] \times [0,1]$, then $\beta(s,t)$ is a Lagrangian subspace of $\cH_\boxplus$, and the map $(s,t) \mapsto \beta(s,t)$ is smooth.
\end{lemma}

\begin{proof} Since $\beta(0,0)$ is Lagrangian, it suffices to find a smooth family of selfadjoint operators $A(s,t)\colon \beta(0,0) \to \beta(0,0)$ such that $\beta(s,t)$ is the graph of $A(s,t)$, i.e.
\[
	\beta(s,t) = \left\{ z +  J_\boxplus A(s,t)z : z \in \beta(0,0) \right\},
\]
where
\[
J_\boxplus (x,\phi,y,\psi) = (R^{-1} \phi, - Rx, -R^{-1} \psi, R y)
\]
and
\[
R \colon \Hp \to \Hm  \cong (\Hp)^*
\] 
is the Riesz duality isomorphism (cf. equation (17) in \cite{CJM14}).

It suffices to choose
\[
	A(s,t)(x,0,0,\psi) = - \left(R^{-1}(t\psi + s \mathcal{J} x), 0, 0, tRx \right),
\]
which is selfadjoint because the composition $R^{-1} \circ \mathcal{J} \colon\Hp\to\Hp$ is selfadjoint.

To prove the selfadjointness of $R^{-1} \circ \mathcal{J}$ we use the identity $R^{-1} = R^*$ to compute
\[
	\left<R^{-1}  \mathcal{J} f, g\right>_{\Hp} = \left< \mathcal{J} f, Rg\right>_{\Hm} = ( \mathcal{J} f)g = \int_\Sigma fg
\]
for any $f,g \in \Hp$, where $( \mathcal{J} f)g$ denotes the action of the functional $ \mathcal{J} f \in \Hm$ on $g \in \Hp$. Since the right-hand side of the above equality is symmetric in $f$ and $g$, we obtain $\left<R^{-1}  \mathcal{J} f, g\right>_{\Hp} = \left<R^{-1}  \mathcal{J} g, f\right>_{\Hp}$ as required.
\end{proof}

\begin{lemma}
If $(s,t) \in (-\infty,0] \times [0,1]$, then $\beta(s,t)$ and $\mu(0)$ are a Fredholm pair.
\end{lemma}

\begin{proof}
As in the proof of Lemma \ref{propFredholm}, it suffices to have an estimate of the form
\begin{align}\label{stenergy}
	\|(u_1,u_2)\|_{H^1} \leq C \left( \|P_\beta^\bot \tr(u_1,u_2) \|_{\cH_\boxplus} + \|(u_1,u_2)\|_{L^2} \right)
\end{align}
for all $u_i \in K_i^0$.

We first decompose an arbitrary element $\tr (u_1,u_2) \in \mu(0)$ into
\[
	P_\beta \tr(u_1,u_2) = (x, t\phi + s  \mathcal{J} x, tx, \phi), \quad P_\beta^\bot \tr(u_1, u_2) = (-ty - s \mathcal{J}^*\psi, \psi, y, -t\psi) 
\]
for some $x,y \in \Hp$ and $\phi, \psi \in \Hm$, and observe that
\begin{align*}
	\| P_\beta \tr(u_1,u_2) \|^2_{\cH_\boxplus} &\geq \|x\|_{\Hp}^2 + \|\phi\|_{\Hm}^2 \\
	\| P^\bot_\beta \tr(u_1,u_2) \|^2_{\cH_\boxplus} &\geq \|y\|_{\Hp}^2 + \|\psi\|_{\Hm}^2.
\end{align*}
We then compute
\begin{align*}
	\int_\Sigma u_1 \frac{\p u_1}{\p \nu_1} + u_2 \frac{\p u_2}{\p \nu_2} &= (t\phi + s \mathcal{J} x + \psi)(x-ty-s \mathcal{J}^*\psi) - (\phi - t\psi)(tx+y) \\
	&\leq (1+t^2) \left[ \psi(x) - \phi(y)\right] - s\left[  \mathcal{J} x(ty + s \mathcal{J}^*\psi) + t(\phi+\psi)( \mathcal{J}^*\psi)\right]
\end{align*}
using the fact that $s \mathcal{J} x(x) = s \|x\|_{L^2(\Sigma)}^2 \leq 0$.
Using the arithmetric--geometric mean inequality on all but the last term on the right-hand side, we obtain
\begin{align*}
	\int_\Sigma u_1 \frac{\p u_1}{\p \nu_1} + u_2 \frac{\p u_2}{\p \nu_2}
	& \leq C \left( \epsilon \| P_\beta \tr(u_1,u_2) \|^2_{\cH_\boxplus} + \epsilon^{-1} \| P^\bot_\beta \tr(u_1,u_2) \|^2_{\cH_\boxplus} \right) - st \psi( \mathcal{J}^*\psi)
\end{align*}
for some constant $C = C(s,t)$. Finally, we note that $\psi( \mathcal{J}^*\psi) \leq \| \mathcal{J}^*\| \|\psi\|_{\Hm}^2 \leq \| \mathcal{J}^*\| \| P^\bot_\beta \tr(u_1,u_2) \|^2_{\cH_\boxplus}$ and choose $\epsilon$ sufficiently small, and the estimate \eqref{stenergy} follows.
\end{proof}

We next observe that $\beta(s,t)$ and $\mu(0)$ are disjoint for sufficiently negative $s$; this is a consequence of a uniform lower bound on the Dirichlet-to-Neumann operators.

\begin{lemma}\label{slowerbound}
There exists $s_\infty < 0$ such that $\beta(s,t) \cap \mu(0) = \{0\}$ for any $s \leq s_\infty$ and $t \in [0,1]$.
\end{lemma}

\begin{proof}
From the definition of $\mu$ we have
\[
	\mu(0) = \left\{ (f_1, \tL_1 f_1, f_1, -\tL_2 f_2) : f_1,f_2 \in \Hp \right\},
\]
and so if $\beta(s,t) \cap \mu(0) \neq \varnothing$ there exists a function $f \in \Hp$ that satisfies $(\tL_1 + t^2 \tL_2) f = s \cJ f$. For $i \in \{1,2\}$ let $u_i \in K_i^0$ denote the unique solution to $Lu_i = 0$ with $\left. u_i \right|_\Sigma = f$. Integrating by parts, we have
\begin{align*}
	s \|f\|_{L^2(\Sigma)}^2 
	&= \int_\Sigma \left( u_1 \frac{\p u_1}{\p \nu_1} + t^2 u_2 \frac{\p u_2}{\p \nu_2}\right) \\
	&= \int_{\Omega_1} \left[ |\nabla u_1|^2 + V u_1^2\right] + t^2 \int_{\Omega_2} \left[ |\nabla u_2|^2 + V u_2^2\right], 
\end{align*}
so there exists a positive constant $C$, independent of $s$ and $t$, such that
\begin{align}\label{sjbound}
	\| (u_1, u_2)\|_{H^1}^2 \leq C \|(u_1,u_2)\|_{L^2}^2 + s\|f\|^2_{L^2(\Sigma)}. 
\end{align}

Now suppose the conclusion of the lemma is false, so there exist sequences of real numbers $s_j \in (-\infty,0]$ and $t_j \in [0,1]$, and functions $f_j \in \Hp$, such that $s_j \to -\infty$ and $(\Lambda_1 + t_j^2 \Lambda_2)f_j = s_j f_j$. Let $u_{1j}$ and $u_{2j}$ denote the unique functions in $K_1^0$ and $K_2^0$ that satisfy $\left. u_{1j} \right|_\Sigma = \left. u_{2j} \right|_\Sigma = f_j$. Without loss of generality we assume that
\[
	\|(u_{1j}, u_{2j})\|_{L^2}^2 = \|u_{1j}\|_{L^2(\Omega_1)}^2 + \|u_{2j}\|_{L^2(\Omega_2)}^2 = 1.
\]
Since $s_j \leq 0$, \eqref{sjbound} implies $\{u_{1j}\}$ and $\{u_{2j}\}$ are bounded in $H^1$, so there are subsequences with
\[
	(u_{1j},u_{2j}) \to (\bar{u}_1,\bar{u}_2) \text{ in } L^2, \quad
	(u_{1j},u_{2j}) \rightharpoonup (\bar{u}_1,\bar{u}_2) \text{ in } H^1.
\]
It follows that $\bar{u}_1 \in K_1^0$ and $\bar{u}_2 \in K_2^0$, with $\left. \bar{u}_1 \right|_\Sigma = \left. \bar{u}_2 \right|_\Sigma \in \Hp$. Moreover, the compactness of the embedding $\Hp \hookrightarrow L^2(\Sigma)$ implies $\left. u_{1j} \right|_\Sigma \to \left. \bar{u}_1 \right|_\Sigma$ in $L^2(\Sigma)$. However, since $s_j \leq 0$ and $s_j \to -\infty$, \eqref{sjbound} implies $\left. u_{1j} \right|_\Sigma \to 0$ in $L^2(\Sigma)$, hence $\left. \bar{u}_1 \right|_\Sigma = 0$. Since $0 \notin \sigma(L^D_1)$, this implies $\bar{u}_1 = 0$. We similarly find that $\bar{u}_2 = 0$, which contradicts the fact that
\[
	\|\bar{u}_1\|_{L^2(\Omega_1)}^2 + \|\bar{u}_2\|_{L^2(\Omega_2)}^2 = 1
\]
and thus completes the proof.
\end{proof}

\begin{rem}\label{addreg}
If we assume that the metric tensor $g$ is Lipschitz, instead of just $L^\infty$, the above compactness argument is not needed. With this additional regularity hypothesis, Theorem 4.25 of \cite{M00} gives
\[
	\|(u_1,u_2) \|_{L^2} \leq C \|f\|_{L^2(\Sigma)}
\]
which, together with \eqref{sjbound}, immediately establishes Lemma \ref{slowerbound}.
\end{rem}

We have thus shown that $\beta(s,t)$ is a smooth curve in the Fredholm--Lagrangian Grassmannian $\cF \Lambda_{\mu(0)}(\cH_\boxplus)$, so it has a well-defined Maslov index with respect to either $s$ or $t$.


\begin{prop}\label{betamorse}
If $t _0 \in [0,1]$ is fixed, then $\Mas(\beta(s,t_0); \mu(0)) = \Mor_0(\Lambda_1 + t_0^2 \Lambda_2)$.
\end{prop}

\begin{proof}
As in the proof of Lemma \ref{slowerbound}, we have that $\beta(s,t) \cap \mu(0) \neq \varnothing$ if and only if there is a function $f \in \Hp$ that satisfies $(\tL_1 + t^2 \tL_2) f = s \cJ f$.  This implies $\dim[\beta(s,t) \cap \mu(0)] = \dim \ker (\tL_1 + t^2 \tL_2 - s\cJ) = \dim \ker (\Lambda_1 + t^2 \Lambda_2 - s)$, where the last equality is a consequence of Lemma \ref{DNdomains}.

We claim that the path $s \mapsto \beta(s,t_0)$ is positive definite. Assuming the claim, it follows from Lemma \ref{slowerbound} that
\[
	\Mas(\beta(s,t_0); \mu(0)) = \sum_{s \leq 0} \dim[\beta(s,t_0) \cap \mu(0)]
	= \Mor_0(\Lambda_1 + t_0^2 \Lambda_2).
\]

To prove the claim, suppose that $s_*$ is a crossing time, so there is a path $z(s) = (x, t_0 \phi + sx, t_0 x, \phi)$ in $\beta(s)$ with $z(s_*) \in \beta(s_*,t_0) \cap \mu(0)$. We compute
\begin{align*}
	\omega_\boxplus\left( z, \frac{dz}{ds} \right) &= \omega_\boxplus\left( (x, t_0\phi + sx, t_0x, \phi), (0, x, 0, 0) \right) \\
	&= \omega((x,t_0\phi),(0,x)) - \omega((t_0 x,\phi),(0,0)) \\
	&= \| x \|_{L^2(\Sigma)}^2,
\end{align*}
which is positive unless $x=0$. But if $x=0$, then $t_0\phi + s_* x = \tL_1x = 0$, which is not possible because $z(s_*) \neq 0$, so the claim is proved.
\end{proof}

\begin{proof}[Proof of Theorem \ref{DNcrossing}]
Since the boundary of $\beta\colon [s_\infty,0] \times [0,1] \to \cF_{\mu(0)} \Lambda(\cH_\boxplus)$ is null-homotopic, its Maslov index vanishes, hence
\[
	\Mas(\beta(s,0); \mu(0)) + \Mas(\beta(0,t); \mu(0)) = \Mas(\beta(s,1); \mu(0)).
\]
Since $\beta(0,t) = \beta(t)$, Proposition \ref{betamorse} implies
\[
	\Mas(\beta(t); \mu(0)) = \Mor_0(\Lambda_1 + \Lambda_2) - \Mor_0(\Lambda_1).
\]
and the proof is complete.
\end{proof}

%
%
%

%

\begin{lemma} Suppose $0 \notin \sigma(L^D_1) \cup \sigma(L^D_2)$, so $\Lambda_1$ and $\Lambda_2$ are well defined. Then
\[
	\Mor(\Lambda_1 + \Lambda_2) \leq \Mor(\Lambda_1) + \Mor(\Lambda_2).
\]
and
\[
	\Mor_0(\Lambda_1 + \Lambda_2) \leq \Mor_0(\Lambda_1) + \Mor_0(\Lambda_2).
\]
\end{lemma}
In fact the proof yields the stronger result that $\Mor(\Lambda_1 + \Lambda_2)$ is bounded above by the dimension of the sum of the negative subspaces for $\Lambda_1$ and $\Lambda_2$, which is  bounded above by the sum of the dimensions.

\begin{proof} Since $\Lambda_1$ is selfadjoint with compact resolvent, there is a spectral decomposition $L^2(\Sigma) = E_1^- \oplus E_1^0 \oplus E_1^+$, and similarly for $\Lambda_2$. Let $p = \dim(E_1^-)$ and $q = \dim(E_2^-)$. Then $V := E_1^- + E_2^-$ has dimension $r \leq p+q$, and $\Lambda_1 + \Lambda_2$ is nonnegative on $V^\bot$. Therefore
\[
	\sup_{\dim(U) = r} \inf \left\{ \frac{\left<(\Lambda_1 + \Lambda_2)f,f\right> }{\|f\|_{L^2(\Sigma)}}: f \in U^\bot  \right\} \geq
	\inf \left\{ \frac{\left<(\Lambda_1 + \Lambda_2)f,f\right> }{\|f\|_{L^2(\Sigma)}}: f \in V^\bot  \right\} \geq 0
\]
and the minimax principle implies $\lambda_{r+1}(\Lambda_1 + \Lambda_2) \geq 0$, so $\Mor(\Lambda_1 + \Lambda_2) \leq r$ and the proof is complete.

Replacing $E_i^-$ by $E_i^- \oplus E_i^0$ for $i\in \{1,2\}$ yields $\lambda_{r+1}(\Lambda_1 + \Lambda_2) > 0$, and the second inequality follows.
\end{proof}

%
%


\section{Applications}\label{secApp}
We now present several applications of the results proved above. The results in Sections \ref{secDN} and \ref{secPeriodic} are not immediate consequences of Theorems \ref{MorMas} and \ref{DNcrossing} but follow from similar constructions, so we only sketch the proofs.

\subsection{The nodal deficiency}
Our first application is Corollary \ref{nodalDef}, which gives an explicit formula for the nodal deficiency of an eigenfunction $\phi_k$ corresponding to the simple eigenvalue $\lambda_k$. 

\begin{proof}[Proof of Corollary \ref{nodalDef}]
Define
\[
	\Omega_+ = \{ \phi_k > 0\}, \quad \Omega_- = \{ \phi_k < 0\},
\]
and let $n_\pm(\phi_k)$ denote the number of connected components of $\Omega_\pm$, so the total number of nodal domains is $n(\phi_k) = n_+(\phi_k) + n_-(\phi_k)$.

For $\epsilon > 0$ define $L(\epsilon) = -\Delta - (\lambda_k + \epsilon)$. On each nodal domain the first Dirichlet eigenvalue of $L$ is $\lambda_k$, so for small enough $\epsilon>0$ the Dirichlet realizations $L^D_\pm(\epsilon)$ are invertible, with
\[
	\Mor(L^D_\pm(\epsilon)) = n_\pm(\phi_k).
\]
Moreover, the global realization $L^G(\epsilon)$ is invertible as long as $\lambda_k + \epsilon < \lambda_{k+1}$, and so
\[
	\Mor(L^G(\epsilon)) = k.
\]

Now let $\Lambda_\pm(\epsilon)$ denote the Dirichlet-to-Neumann maps for $L(\epsilon)$ on $\Omega_\pm$. It follows immediately from Corollary \ref{DNbracket} that the nodal deficiency $\delta(\phi_k) = k - n(\phi_k)$ is given by
\[
	\delta(\phi_k) = \Mor \left( \Lambda_+(\epsilon) + \Lambda_-(\epsilon) \right),
\]
and so the proof is complete.
\end{proof}

\subsection{Almost doubled manifolds}\label{secDoubled}
We now consider the setting of Theorem \ref{perturb}, in which
\begin{align}\label{DNbound}
	\left\| \tL_1^{-1} \tL_2 - c I \right\|_{B(\Hp)} < 1+c
\end{align}
for some constant $c$. 

\begin{proof}[Proof of Theorem \ref{perturb}]
As in the proof of Proposition \ref{betamorse}, we have that $t_*$ is a crossing time, i.e. $\beta(t_*) \cap \mu(0) \neq \{0\}$, if and only if $\ker(I + t_*^2 \tL_1^{-1} \tL_2)$ is nontrivial. We compute
\begin{align*}
	I + t^2 \tL_1^{-1} \tL_2
	&= I +ct^2I+  t^2 \tL_1^{-1} \tL_2 - ct^2 I \\
	&= (1 + c t^2) \left[I + \frac{t^2}{1+ct^2} \left(  \tL_1^{-1} \tL_2 - c I\right) \right],
\end{align*}
hence $I + t_*^2 \tL_1^{-1} \tL_2$ is invertible if
\[
	\frac{t_*^2}{1+c t_*^2} \left\| \tL_1^{-1} \tL_2 - c I \right\|_{\Hp} < 1
\]
Since the function $t \mapsto t^2/(1+ct^2)$ is increasing, it suffices to verify the above condition at $t=1$. This is just the inequality \eqref{DNbound}, so the result follows.
\end{proof}

A simple case is when $\tau\colon M \to M$ is an involution such that $\left.\tau\right|_\Sigma = \operatorname{id}$, $\tau(\Omega_1) = \Omega_2$, and $L(u\circ \tau) = (Lu) \circ \tau$ for all $u \in H^1(M)$. If $\p M$ is nonempty, it is necessary to assume that $\tau(\Sigma_D) = \Sigma_D$ and $\tau(\Sigma_N) = \Sigma_N$. For instance, if the cylinder shown in Figure \ref{double} is given by $[0,2\pi] \times \bbS^1$, with the involution $\tau(x,\theta) = (2\pi-x,\theta)$, we require either $\p M = \Sigma_D$ or $\p M = \Sigma_N$, so both $\{0\} \times \bbS^1$ and $\{2\pi \} \times \bbS^1$ have the same boundary conditions. It follows immediately that $\tL_1 = \tL_2$, so Theorem \ref{perturb} implies $\Mas(\beta(t); \mu(0)) =0$.

If the involution $\tau$ does not preserve the boundary conditions, it may not be the case that $\tL_1 = \tL_2$, and $\Mas(\beta(t); \mu(0))$ may be nonzero. This can be seen from an elementary computation for the operator $L = -(d/dx)^2-C$ on $[0,\ell]$, with $C$ a positive constant. We let $\Omega_1 = (0,\ell/2)$ and $\Omega_2 = (\ell/2,\ell)$, and compute $\Mas(\beta(t); \mu(0))$ assuming Dirichlet boundary conditions at $x=0$ and Neumann conditions at $x=\ell$.

%

A basis for the weak solution space of the equation $Lu=0$ is
\[
	\sin(\sqrt{C}x), \quad \cos(\sqrt{C}(x-\ell))
\]
and so the space of two-sided Cauchy data at $\ell/2$ (computed according to \eqref{mudef}) is
\[
	\mu(0) = \left\{ \left( a\sin(\sqrt{C}\ell/2), a\sqrt{C} \cos(\sqrt{C}\ell/2), b\cos(\sqrt{C}\ell/2), b\sqrt{C}\sin(\sqrt{C}\ell/2) \right) : a,b \in \bbR \right\}.
\]
A crossing occurs at time $t_*$ precisely when $a=b$ and $t_* = \cot(\sqrt{C}\ell/2)$, so there is at most one crossing time in $[0,1]$. In particular, there is a crossing if and only if
\begin{align}\label{Ccondition}
	\frac{\pi}{4} + n\pi \leq \frac{\sqrt{C}\ell}{2} \leq \frac{\pi}{2} + n\pi
\end{align}
for some integer $n\geq 0$. The crossing form is given by
\[
	2 a^2 \sqrt{C} \sin^2(\sqrt{C}\ell/2) > 0,
\]
and so the Maslov index is either 1 or 0, depending on whether or not \eqref{Ccondition} is satisfied.

\subsection{The Dirichlet-to-Neumann map}\label{secDN}
We next use the Maslov index to prove Theorem \ref{thmFriedlander}. A geometric proof of this result was given by Mazzeo in \cite{M91}; here we observe that Mazzeo's proof can be formulated in terms of the Maslov index.

The na\"ive idea is to define a one-parameter family of Lagrangian subspaces that moves between $\{0\} \oplus \Hm$ and $\Hp \oplus \{0\}$, which correspond to Dirichlet and Neumann boundary conditions, respectively. By a homotopy argument the Maslov index of this path equals the difference of the Dirichlet and Neumann Morse indices of $L$. On the other hand, a direct computation shows that the Maslov index equals the Morse index of the Dirichlet-to-Neumann map $\Lambda$. This depends on a monotonicity property for the eigenvalues of $\Lambda$, which has a natural interpretation via the Maslov index.

However, this approach suffers from the fact that the path of Lagrangian subspaces that interpolates between Dirichlet and Neumann boundary conditions
fails to be continuous at the Dirichlet endpoint; see Remark \ref{cont}. To overcome this obstacle, we interpolate between Neumann and ``almost Dirichlet" boundary conditions, then use asymptotic results for the Robin boundary value problem to relate the Dirichlet and almost Dirichlet spectra.

Since $i \in \{1,2\}$ is fixed, we restrict our attention to a single domain, which we call $\Omega$, with Lipschitz boundary $\Sigma$. We first define the subspace
\begin{align}\label{betaDN}
	\beta(t) = \{(x,\phi) \in \cH : (\cos t) \cJ x + (\sin t)\phi = 0\}
\end{align}
in $\cH = \Hh$, where $\cJ\colon \Hp \to \Hm$ denotes the compact inclusion. We let $K^\lambda \subset H^1(\Omega)$ denote the space of weak solutions to $Lu = \lambda u$ and define the space of Cauchy data
\[
	\mu(\lambda) = \left\{ \left.\left(u,\frac{\p u}{\p \nu} \right)\right|_\Sigma : u \in K^\lambda \right\}.
\]
Note that $\mu(0)$ is the graph of $\tL\colon\Hp\to\Hm$.

\begin{lemma} \label{betabound}
There exists $\lambda_\infty < 0$ such that $\beta(t) \cap \mu(\lambda) = \{0\}$ for every $t \in (0,\pi/2]$ and $\lambda \leq \lambda_\infty$.
\end{lemma}

\begin{proof} Suppose $\beta(t) \cap \mu(\lambda) \neq \{0\}$. By definition there exists a function $u \in H^1(\Omega)$ such that $Lu = \lambda u$ weakly and
\[
	(\cos t) \left.u\right|_\Sigma + (\sin t) \left.\frac{\p u}{\p\nu}\right|_\Sigma = 0,
\]
hence
\[
	\int_\Sigma u \frac{\p u}{\p \nu} = - (\cot t) \int_\Sigma u^2 \leq 0
\]
because $\cot t \geq 0$ for $t \in (0,\pi/2]$.
Green's formula implies
\[
	\lambda \int_\Omega u^2 = - \int_\Sigma u \frac{\p u}{\p \nu} + \int_\Omega \left[ |\nabla u|^2 + V(x) u^2 \right]  \geq \int_\Omega \left[ |\nabla u|^2 + V(x) u^2 \right]
\]
and the result follows with any $\lambda_\infty < \inf V(x)$.
\end{proof}

\begin{prop} There exists $\epsilon > 0$ such that the subspaces $\mu(\lambda)$ and $\beta(t)$ are smooth, Lagrangian, and form a Fredholm pair for $(\lambda,t) \in [\lambda_\infty,0] \times [\epsilon,\pi/2]$. Moreover,
\begin{align}
	\Mas(\mu(\lambda); \beta(\epsilon)) &= -\Mor(L^D) \label{betaDir} \\
	\Mas(\mu(\lambda); \beta(\pi/2)) &= -\Mor(L^N) \label{betaNeu} \\
	\Mas(\beta(t); \mu(\lambda_\infty)) &= 0 \label{betainf} \\
	\Mas(\beta(t); \mu(0)) &= \Mor_0(\Lambda). \label{betaLambda}
\end{align}
\end{prop}

Referring to the square in Figure \ref{DNhomotopy}, the paths in \eqref{betaDir}--\eqref{betaLambda} correspond to the bottom, top, left and right sides, respectively.

\begin{center}
\begin{figure}
\begin{tikzpicture}
	\draw[thick,->] (-2,0) -- (4,0); 
	\node at (3.5,0.4) {$t=\epsilon$}; 
	\node at (4.5,0) {$\lambda$}; 
	\node at (3.7,3) {$t=\pi/2$}; 
	
	\draw[thick, ->] (3,-1) -- (3,4); 
	\node at (3,4.5) {$t$}; 
	\node at (0,-0.3) {$\lambda = \lambda_\infty$}; 
		
	\draw[very thick] (0,0.4) -- (0,3); 
	\draw[very thick] (0,0.4) -- (3,0.4); 
	\draw[very thick] (0,3) -- (3,3); 
	\draw[very thick] (3,0.4) -- (3,3); 
	
	\filldraw (1.5,0) circle (2pt); 
	\filldraw (2.5,0) circle (2pt); 
	\filldraw (0.5,3) circle (2pt); 
	\filldraw (1.5,3) circle (2pt); 
	\filldraw (2,3) circle (2pt); 
	\filldraw (3,1.5) circle (2pt); 
	
	\draw[thick] (0.5,3) to [out=270, in=90] (1.5,0); 
	\draw[thick] (1.5,3) to [out=270, in=90] (2.5,0); 
	\draw[thick] (2,3) to [out=270, in=150] (3,1.5); 
	
%
\end{tikzpicture}
\caption{An illustration of the homotopy in the proof of Theorem \ref{thmFriedlander}. A crossing at $t=0$ or $t=\pi/2$ occur when $\lambda$ is an eigenvalue for the Dirichlet or Neumann problem, respectively. A crossing at $\lambda=0$ occurs when $-\cot t$ is an eigenvalue of the Dirichlet-to-Neumann map. The assumption $0 \notin \sigma(L^D)$ precludes the existence of a crossing at the origin.}
\label{DNhomotopy}
\end{figure}
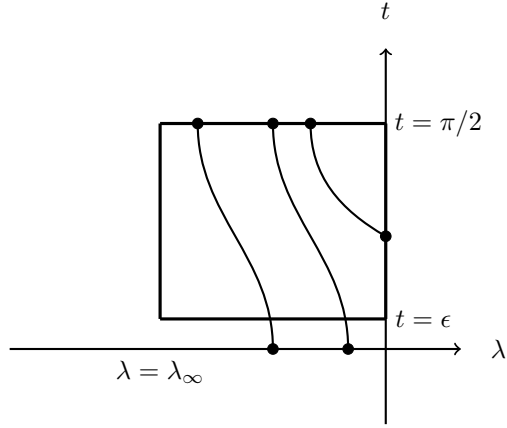
\end{center}

\begin{proof}
That $\mu(\lambda)$ is a smooth curve of Lagrangian subspaces was established in Proposition 3.5 of \cite{CJM14}.  For $t \neq 0$ we can express $\beta(t)$ as the graph of the family of selfadjoint operators $A(t) = (\cot t) R^{-1} \cJ$ on the Lagrangian subspace $\rho = \Hp$, where $R\colon \Hp\to\Hm$ is the Riesz duality operator. That is,
\[
	\beta(t) = G_\rho(A(t)) := \{x + JA(t)x : x \in \rho \},
\]
with $J(x,\phi) := (R^{-1}\phi, -Rx)$. It follows (cf. the proof of Lemma \ref{betaLag}) that $\beta(t)$ is a smooth family of Lagrangian subspaces for $t \neq 0$.

We next show that $\mu(\lambda)$ and $\beta(t)$ comprise a Fredholm pair, first observing that the result is already known for $t=\pi/2$ (the Neumann case) by Lemma 3.8 of \cite{CJM14}. Let $P$ denote the orthogonal projection onto $\beta(\pi/2) = \Hp$ and define $A(t) = I - (\cot t) \cJ P \colon \cH \to \cH$. Since $A(t) \beta(\pi/2) = \beta(t)$ and $A(t)$ is a compact perturbation of the identity, the result follows from Lemma 3.2 in \cite{CJLS14}.

The monotonicity of $\mu$ with respect to $\lambda$, as was shown in Lemma 4.2 of \cite{CJM14}, immediately yields \eqref{betaNeu} and the equality $\Mas(\mu(\lambda); \beta(\epsilon)) = -\Mor(L^R_\epsilon)$, where $L^R_\epsilon$ is the realization of $L$ with $\beta(\epsilon)$ boundary conditions,
\[
	\frac{\p u}{\p\nu} = -(\cot \epsilon) u.
\]
We now apply Proposition $3$ of \cite{arendt2012friedlander}, which gives that the ordered eigenvalues of $L^R_\epsilon$ converge to the ordered eigenvalues of $L^D$ as $\epsilon \to 0$. Since $0 \notin \sigma(L^D)$ this implies $\Mor(L^R_\epsilon) = \Mor(L^D)$ for sufficiently small $\epsilon$, and \eqref{betaDir} follows.

Equality \eqref{betainf} is an immediate consequence of Lemma \ref{betabound}, so only \eqref{betaLambda} remains.

Suppose $t_* \in [\epsilon, \pi/2]$ is a crossing time, so $\beta(t_*) \cap \mu(0) \neq \{0\}$. Then there exists a path $z(t) = (x(t), \phi(t))$ in $\cH$ such that $z(t) \in \beta(t)$ for $|t - t_*| \ll 1$ and $z(t_*) \in \mu(0)$. Since $t_* \neq 0$ we have $\phi(t) = -(\cot t) \cJ x(t)$. It follows that
\[
	\phi'(t) = (\csc^2 t)\cJ x(t) - (\cot t) \cJ x'(t)
\]
and so the crossing form
\[
	Q(z(t_*)) = \left. \omega(z,z')\right|_{t=t_*} = (\csc^2 t_*) \| x(t_*) \|_{L^2(\Sigma)}^2
\]
is strictly positive. Since $\beta(t) \cap \mu(0) \neq \{0\}$ if and only if $-(\cot t)$ is an eigenvalue of $\Lambda$, we conclude that
\[
	\Mas(\beta(t),\mu(0)) = \sum_{t \in (\epsilon,\pi/2]} \dim \left[ \beta(t) \cap \mu(0) \right] 
	= \# \sigma(\Lambda) \cap (-\cot \epsilon, 0] = \Mor_0(\Lambda)
\]
for $\epsilon > $ sufficiently small.
\end{proof}

Having established \eqref{betaDir}--\eqref{betaLambda}, Theorem \ref{thmFriedlander} follows from a homotopy argument as in the proof of Theorem \ref{MorMas}. We conclude the section by justifying the decision to restrict the path $\beta(t)$ to the interval $[\epsilon,\pi/2]$.

\begin{rem}\label{cont}
 The path $\beta(t)$ defined in \eqref{betaDN} is discontinuous at $t=0$. By definition this means the corresponding family of orthogonal projections $P(t)$ is discontinuous. It is easy to see that $S(t) (x,\phi) := x - (\cot t) \cJ x$ defines a projection of $\cH$ onto $\beta(t)$ for $t \neq 0$. Using Lemma 12.8 of \cite{BW93} we compute the orthogonal projection
\begin{align*}
	P(t) &= S S^* \left[SS^* + (I-S^*)(I-S)\right]^{-1} \\ &= 
	\left( \begin{matrix} I_{\Hp} & -(\cot t) \cJ^* \\ -(\cot t) \cJ & (\cot^2 t) \cJ \cJ^* \end{matrix} \right)
	\left( \begin{matrix} \left[I_{\Hp} + (\cot^2 t) \cJ^*\cJ \right]^{-1} & 0 \\ 0 & \left[I_{\Hm} + (\cot^2 t) \cJ\cJ^* \right]^{-1} \end{matrix} \right).
\end{align*}
Now consider the component in the upper left-hand corner, 
\[
	P_{11}(t) := \left[ I_{\Hp} + (\cot^2 t) \cJ^* \cJ \right]^{-1} \colon\Hp\lra\Hp
\]
for $t \neq 0$. Since $P_{11}(0) = 0$, it suffices to prove that $\|P_{11}(t)\| \nrightarrow 0$. The operator $\cJ^*\cJ \colon\Hp\to\Hp$ is compact and injective (because $\cJ$ is), and hence has a sequence of eigenvalues tending to zero. Thus there exists $\{f_n\}$ in $\Hp$ such that $\| \cJ^*\cJ f_n\|_{\Hp} \leq n^{-1} \| f_n\|_{\Hp}$. Letting $g_n = f_n + (\cot^2 t) \cJ^* \cJ f_n$, we have $P_{11}(t) g_n = f_n$ and $\|g_n\|_{\Hp} \leq \left[1 + (\cot^2 t)/n  \right] \|f_n\|_{\Hp}$, therefore
\[
	\| P_{11}(t) g_n \|_{\Hp} = \|f_n\|_{\Hp} \geq \frac{\| g_n \|_{\Hp}}{1 + (\cot^2 t) /n}.
\]
Letting $n \to \infty$ implies $\|P_{11}(t)\| \geq 1$ for each $t \neq 0$, hence $P_{11}(t) \nrightarrow 0$.
\end{rem}


\subsection{Periodic boundary conditions}\label{secPeriodic}

We finally treat the case of a single domain $M$, with boundary $\pM = \Gamma_1 \cup \Gamma_2$ glued to itself via a map $\tau\colon\Gamma_1 \to \Gamma_2$, as described in Theorem \ref{thmPeriodic} and shown in Figure \ref{periodic}. The motivating example is the $n$-torus $\bbT^n$, which can be viewed as a cube in $\bbR^n$ with opposite faces identified. We can similarly consider any compact, orientable surface of genus $g$, which is just a $2g$-gon with opposite faces identified.

%
%

We first establish some analytic properties of the periodic Dirichlet-to-Neumann map $\Lambda_\tau$ defined in \eqref{Ltdef}. Throughout this section we assume the hypotheses of Theorem \ref{thmPeriodic}.

\begin{prop}
The periodic Dirichlet-to-Neumann map $\Lambda_\tau$ on $L^2(\Gamma_1)$ is bounded below and selfadjoint with compact resolvent.
\end{prop}

\begin{proof}
Let $f \in H^{1/2}(\Gamma_1)$. By Theorems 3.23 and 3.40 of \cite{M00}, $\left.f\right|_{\Gamma_1^i} \in H^{1/2}_0(\Gamma^i_1)$ and $\left.f\right|_{\Gamma_1^i} \circ \tau^{-1} \in H^{1/2}_0(\Gamma^i_2)$ for each $i$. Therefore the function given by $f$ on $\Gamma_1$ and $f \circ \tau^{-1}$ on $\Gamma_2$ is contained in $H^{1/2}(\pM)$, so the boundary value problem
\[
	Lu = 0, \quad \left.u\right|_{\Gamma_1} = f, \quad \left.u\right|_{\Gamma_2} = f \circ \tau^{-1}
\]
has a unique solution $u \in H^1(M)$. Moreover, $\| u \|_{H^1(M)} \leq C \|f \|_{H^{1/2}(\pM)}$ for some $C>0$ that does not depend on $f$, and so
\begin{align}\label{Qdef}
	Q(f) = \int_M \left[ |\nabla u|^2 + Vu^2 \right]
\end{align}
defines a bounded quadratic form on $H^{1/2}(\pM)$. Since the trace map $H^1(M) \to H^{1/2}(\pM)$ is bounded and $V \in L^\infty(M)$, there are positive constants $c_1$ and $c_2$ such that
\[
	Q(f) \geq c_1 \|f\|_{H^{1/2}(\Gamma_1)} - c_2 \|u\|_{L^2(M)}.
\]
By a standard compactness argument, there exists for any $\epsilon > 0$ a constant $c_3 > 0$ such that
\[
	\|u\|_{L^2(M)}^2 \leq \epsilon \|u\|_{H^1(M)}^2 + c_3 \left\| \left.u\right|_{\pM} \right\|_{L^2(\Gamma_1)}^2
\]
for any $u \in H^1(M)$. It follows that
\[
	Q(f) \geq c_1' \|f\|_{H^{1/2}(\Gamma_1)} - c_2' \|f\|_{L^2(\Gamma_1)}
\]
for every $f \in H^{1/2}(\Gamma_1)$, so $Q$ is bounded over $H^{1/2}(\Gamma_1)$ and coercive over $L^2(\Gamma_1)$.

With $B$ denoting the bilinear form corresponding to $Q$, there exists a selfadjoint operator $T$, with domain $\cD(T) \subset H^{1/2}(\Gamma_1)$, such that $B(f,g) = \left<Tf,g\right>_{L^2(\Gamma_1)}$ for all $f \in \cD(T)$ and $g \in H^{1/2}(\Gamma_1)$. Integrating by parts and using the hypothesis $\tau^* d\mu_2 = d\mu_1$, we find that
\[
	B(f,g) = \int_{\Gamma_1} g \frac{\p u}{\p \nu} + \int_{\Gamma_2} (g \circ \tau^{-1}) \frac{\p u}{\p \nu} 
	=  \int_{\Gamma_1} g \left( \frac{\p u}{\p \nu} + \frac{\p u}{\p \nu} \circ \tau \right) 
\]
for any $g \in H^{1/2}(\Gamma_1)$, hence
\begin{align}\label{Ltdef1}
	T f = \left.\frac{\p u}{\p \nu}\right|_{\Gamma_1} + \left.\frac{\p u}{\p \nu} \right|_{\Gamma_2}\circ \tau.
\end{align}
Therefore $T = \Lambda_\tau$, as defined in \eqref{Ltdef}, and the proof is complete.
\end{proof}

We now sketch the proof of Theorem \ref{thmPeriodic}, which closely follows the proofs of Theorems \ref{MorMas}, \ref{DNcrossing} and \ref{thmFriedlander}.

\begin{proof}[Proof of Theorem \ref{thmPeriodic}]
We first define the symplectic Hilbert spaces $\cH_1 = H^{1/2}(\Gamma_1) \oplus H^{-1/2}(\Gamma_1)$, with the usual symplectic form $\omega_1$, then let $\cH_p = \cH_1 \oplus \cH_1$, with the form $\omega_p = \omega_1 \oplus (-\omega_1)$.

By $K^\lambda \subset H^1(M)$ we denote the space of weak solutions to $Lu = \lambda u$ (with no boundary conditions imposed). We define the space of Cauchy data
\[
	\mu(\lambda) = \left\{ \left( \left.\left(u,\frac{\p u}{\p \nu} \right) \right|_{\Gamma_1}, \left.\left(u,-\frac{\p u}{\p \nu} \right)\right|_{\Gamma_2} \circ \tau^{-1} \right) : u \in K^\lambda \right\}
\]
and the path of boundary conditions
\[
	\beta(t) = \{(x,t\phi,tx,\phi) : (x,\phi) \in \cH_1 \}
\]
as in \eqref{mudef} and \eqref{betadef}.

By a (now familiar) homotopy argument we find that
\begin{align}\label{MorP}
	\Mor(L^P) = \Mor(L^{DN}) + \Mas(\beta(t); \mu(0)),
\end{align}
where $L^{DN}$ denotes the ``mixed realization" of $L$ with Neumann conditions on $\Gamma_1$ and Dirichlet on $\Gamma_2$. Next, following the proof of Theorem \ref{DNcrossing}, we define the two-parameter family of Lagrangian subspaces
\[
	\beta(s,t) = \left\{ (x, t\phi + s \mathcal{J} x, tx, \phi) : (x, \phi) \in \cH_1 \right\}.	
\]
and consequently obtain
\begin{align}\label{MorPtau}
	\Mas(\beta(t); \mu(0)) = \Mor_0(\Lambda_\tau) - \Mor_0(\Lambda_1),
\end{align}
where $\Lambda_1$ is the ``partial Dirichlet-to-Neumann map" on $\Gamma_1$, obtain by mapping a function $f$ on $\Gamma_1$ to $\left.\frac{\p u}{\p \nu}\right|_{\Gamma_1}$, where $u$ uniquely solves the boundary value problem
\[
	Lu = 0, \quad \left.u\right|_{\Gamma_1} = f, \quad \left.u\right|_{\Gamma_2} = 0.
\]
Finally, the method of Theorem \ref{thmFriedlander} yields
\begin{align}\label{PFriedlander}
	\Mor(L^{DN}) = \Mor(L^D) + \Mor_0(\Lambda_1).
\end{align}
Combining \eqref{MorP}, \eqref{MorPtau} and \eqref{PFriedlander}, the result follows.
\end{proof}

\bibliographystyle{plain}
\bibliography{maslov}

\end{document}